\documentclass{article}
\usepackage{amsmath}
\usepackage{amssymb}
\usepackage{amsthm}
\usepackage{bbm,mathtools}
\usepackage{enumitem}
\usepackage{enumitem}
\usepackage{hyperref}
\usepackage{theoremref}

\newcounter{dummy}
\newcommand\myitem[1][]{\item[#1]\refstepcounter{dummy}\def\@currentlabel{#1}}

\def\reals{\mathbbm{R}}
\def\ereals{\overline{\reals}}

\def\comp{\mathop{\text{\scriptsize\raise 1pt \hbox{$\circ$}}}}
\def\infconv{\mathop{\text{\scriptsize\raise 1pt \hbox{$\square$}}}}

\def\argmin{\mathop{\rm argmin}\limits}

\def\minimize{\mathop{\rm minimize}\limits}

\def\st{\mathop{\rm subject\ to}}

\def\dom{\mathop{\rm dom}\nolimits}

\def\ovr{\mathop{\rm over}\ }

\def\upto{{\raise 1pt \hbox{$\scriptstyle \,\nearrow\,$}}}
\def\downto{{\raise 1pt \hbox{$\scriptstyle \,\searrow\,$}}}

\def\epi{\mathop{\rm epi}\nolimits}

\def\FF{(\F_t)_{t=0}^T}

\def\B{{\cal B}}

\def\F{{\cal F}}
\def\G{{\cal G}}
\def\H{{\cal H}}

\def\N{{\cal N}}

\newtheorem{theorem}{Theorem}

\newtheorem{remark}[theorem]{Remark}
\theoremstyle{definition}

\theoremstyle{empty}

\title{Dynamic programming and dimensionality in convex stochastic optimization and control\footnote{No data are associated with this article. No external funding was involved and there are no conflicts of interest.}}

\author{Teemu Pennanen\thanks{Department of Mathematics, King's College London, Strand, London, WC2R 2LS, United Kingdom, teemu.pennanen@kcl.ac.uk} \and Ari-Pekka Perkki\"o\thanks{Mathematics Institute, Ludwig-Maximilian University of Munich, Theresienstr. 39, 80333 Munich, Germany, a.perkkioe@lmu.de.}}

\begin{document}

\maketitle

\begin{abstract}
This paper studies stochastic optimization problems and associated Bellman equations in formats that allow for reduced dimensionality of the cost-to-go functions. In particular, we study stochastic control problems in the ``decision-hazard-decision'' form where at each stage, the system state is controlled both by predictable as well as adapted controls. Such an information structure may result in a lower dimensional system state than what is required in more traditional ``decision-hazard'' or ``hazard-decision'' formulations. The dimension is critical for the complexity of numerical dynamic programming algorithms and, in particular, for cutting plane schemes such as the stochastic dual dynamic programming algorithm. Our main result characterizes optimal solutions and optimum values in terms of solutions to generalized Bellman equations. Existence of solutions to the Bellman equations is established under general conditions that do not require compactness. We allow for general randomness but show that, in the Markovian case, the dimensionality of the Bellman equations reduces with respect to randomness just like in more traditional control formulations.
\end{abstract}

\noindent\textbf{Keywords.} Stochastic optimal control, dynamic programming, convexity
\newline
\newline
\noindent\textbf{AMS subject classification codes.} 93E20, 90C39, 26B25

\section{Introduction}

Let $(\Omega,\F,P)$ be a probability space with a filtration $\FF$ and consider the optimal control problem
\begin{equation}\label{oc}
\begin{aligned}
  &\minimize & E\Bigg[\sum_{t=0}^{T-1}&  L_t(X_t,U^b_t,U^a_{t+1})+ J(X_T)\Bigg]\ \ovr\ (X,U^b,U^a)\in\N,\\ 
  &\st\ & X_{t+1} &= F_{t}(X_t,U^b_t,U^a_{t+1})\quad t=0,\dots,T-1\ a.s.,
\end{aligned}
\end{equation}
where the state $X_t$ and the controls $U^b_t$ and $U^a_t$ are random variables taking values in $\reals^N$, $\reals^{M^b}$ and $\reals^{M^a}$. The functions $L_t:\reals^N\times\reals^{M^b}\times\reals^{M^a}\times\Omega\to\ereals$ and $J:\reals^N\times\Omega\to\ereals$ are extended real-valued convex normal integrands and, for each $t$, the function $F_t:\reals^N\times\reals^{M^b}\times\reals^{M^a}\times\Omega\to\reals^N$ is given by
\begin{equation*}
  F_t(X_t,U^b_t,U^a_{t+1},\omega) := A_{t+1}(\omega)X_t+B^b_{t+1}(\omega)U^b_t+B^a_{t+1}(\omega)U^a_{t+1}+W_{t+1}(\omega),
\end{equation*}
where $A_{t+1}$, $B^b_{t+1}$ and $B^a_{t+1}$ are $\F_{t+1}$-measurable random matrices of appropriate dimensions and $W_{t+1}$ are $\F_{t+1}$-measurable random vectors. 
 The set $\N$ denotes the space of adapted stochastic processes, i.e.\ those where $(X_t,U^b_t,U^a_t)$ is $\F_t$-measurable for each $t=0,\ldots,T$. Accordingly, we will assume that $\F_0$ only consists of sets of measure zero or one, so that any $\F_0$-measurable function is almost surely constant.  The constrains in \eqref{oc} are called the {\em system equations}.

Problems of the form \eqref{oc} were introduced in \cite{ccdmt} as a unification of more common formulations of stochastic control problems where one of the control variables $U^b_t$ or $U^a_t$ is absent. While $U^b_t$ represents control decisions taken under the information available at time $t$ only, the control $U^a_{t+1}$ may depend on information available at time $t+1$ after e.g.\ the values of the matrices $A_{t+1}$ and $B_{t+1}$ in the system equations are observed. Problems with controls $U^b_t$ only are sometimes said to be in the ``decision-hazard'' format while problems with $U^a_t$ only are said to be in ``hazard-decision'' format. Accordingly, in \cite{ccdmt} problems of the form \eqref{oc} are said to be in the ``decision-hazard-decision'' format. Such problems arise naturally e.g.\ in ``two-timescale'' stochastic optimization models that are convenient e.g.\ in energy systems modelling; see \cite{pdccjr} and its references.

It is possible to reformulate problem \eqref{oc} in the hazard-decision format by augmenting the system state $X$ by the here-and-now control variable $U^b_t$; see Remark~\ref{rem:hd} below. This, however, increases the dimensionality of the dynamic programming equations by the dimension of $U^b_t$. Dynamic programming algorithms are often sensitive to the dimension of the system state so the state augmentation may be problematic from the computational point of view. The dimension is particularly critical for cutting plane-based algorithms such as stochastic dual dynamic programming (SDDP); see \cite[Section~3.3.1]{nes18}.

It was discovered in \cite[Section~6]{ccdmt} that the extended control formulation of \eqref{oc} allows for dynamic programming equations in the original state space thus avoiding the state augmentation of Remark~\ref{rem:hd}. Indeed, \cite[Theorem~13]{ccdmt} expresses the optimum value of \eqref{oc} as $J_0(X_0)$ where the initial state $X_0$ is fixed and the function $J_0$ is given recursively by the {\em dynamic programming equations} (aka Bellman equations)
\begin{equation}\label{bei}
  J_t(X_t) = \inf_{U^b_t\in\reals^{M^b}}E_t\left[\inf_{U^a_{t+1}\in\reals^{M^a}}\{L_t(X_t,U^b_t,U^a_{t+1})  + J_{t+1}(F_t(X_t,U^b_t,U^a_{t+1}))\}\right].
\end{equation}
Here and in what follows, we use the short hand notation $E_t=E^{\F_t}$ for the $\F_t$-conditional expectation of a random variable. Theorem~13 of \cite{ccdmt} assumed that the randomness is driven by a sequence of stagewise independent random variables and the cost functions $L_t$ were assumed to be ``lower semianalytic'' in order to establish sufficient measurability so that the expectations in \eqref{bei} are well-defined. When the underlying probability space is finite, all the measurability questions disappear. In practice, however, one is often faced with continuously distributed random variables.

This paper extends \cite[Theorem~13]{ccdmt} in the convex case by 1.\ giving optimality conditions for \eqref{oc} in terms of the Bellman functions $J_t$, 2.\ providing sufficient conditions for existence of solutions $(J_t)_{t=0}^T$ to the Bellman equations \eqref{bei}, 3.\ allowing for general randomness in the cost functions and system equations in problem \eqref{oc}. Moreover, we will show that, under an appropriate Markov property, the dimensionality of the cost-to-go functions $J_t$ reduces with respect to the underlying randomness. The last section describes an extension of the Stochastic Dual Dynamic Programming algorithm for \eqref{oc} alluded to at the end of \cite[Section~6]{ccdmt}.

Moreover, we prove the above results in an extended problem format (in Section~\ref{sec:dim}) that is even more general than \eqref{oc} while allowing for Bellman equations with the same dimensionality reductions as for \eqref{oc}. The main results are derived from the theory developed in the recent book \cite{pp24}. Closely related stochastic optimization models were studied in \cite[Section~2.3]{pp24} but without specifying the structure required for dimensionality reduction in the Bellman equations.

The rest of the paper is organized as follows. The following section presents a dynamic programming theory for problem \eqref{oc}. The proofs of the three results in Section~\ref{sec:dpoc} will be given in Section~\ref{sec:oc} as consequences of the dynamic programming theory developed in Section~\ref{sec:dim}. The results in Section~\ref{sec:dim} are, in turn, derived from the results on the abstract Lagrangian problem format in Section~\ref{sec:L}.

\section{Dynamic programming in optimal control}\label{sec:dpoc}

This section gives the main results concerning dynamic programming for problem \eqref{oc}. The proofs will be obtained in Section~\ref{sec:oc} as applications of corresponding results for a more general problem format to be introduced in Section~\ref{sec:dim}.

Much like \cite{rw76} and \cite{evs76}, our approach is based on the theory of {\em normal integrands} developed for optimization problems involving integral functionals; see e.g.\ \cite{roc68}, \cite{bis73}, \cite[Chapter~14]{rw98} and the references there. Recall that a function $h:\reals^n\times\Omega\to\ereals$ is a {\em normal integrand} if the epigraphical mapping $\omega\mapsto\epi h(\cdot,\omega)$ is closed-valued and measurable; see \cite[Chapter~14]{rw98}. Given a sub-$\sigma$-algebra $\G$ of $\F$, the {\em $\G$-conditional expectation} of $h$, when it exists, is a $\G$-measurable normal integrand $E^\G h$ such that 
\[
(E^\G h)(x)= E[h(x)|\G]\ a.s.
\]
for every $\G$-measurable $\reals^n$-valued random variable $x$. Here, $E[h(x)|\G]$ denotes the $\G$-conditional expectation of the random variable $h(x)$. If $h$ is {\em lower bounded} in the sense that there exists $m\in L^1$ such that
\[
h(x,\omega)\ge m(\omega) \quad \forall x\in \reals^n
\]
for almost every $\omega$, then $E^\G h$ exists and is unique; see \cite[Theorem~2.13]{pp24}. If, in addition, the normal integrand $h$ can be expressed as $h(x,\omega)=H(x,\xi(\omega))$ for a $\reals^d$-valued random variable $\xi$ and a normal integrand $H$ on $\reals^n\times\reals^d$ (where we endow $\reals^d$ with the Borel $\sigma$-algebra), then we have the pointwise representation
\begin{equation}\label{eq:kernel}
(E^\G h)(x,\omega) = \int_{\reals^d}H(x,s)\mu(\omega,ds)\quad\forall(x,\omega)\in\reals^n\times\Omega,
\end{equation}
where $\mu$ is the regular $\G$-conditional distribution of $\xi$; see \cite[Lemma~2.55]{pp24}.

When $\G=\F_t$, we will use the short hand notation $E_th=E^{\F_t}h$. We say that two sequences $(J_t)_{t=0}^T$ and $(\tilde Q_t)_{t=0}^T$ of normal integrands solve the {\em Bellman equations} for problem \eqref{oc} if, for almost every $\omega\in\Omega$,
\begin{equation}\label{be}
\begin{split}
J_T(X_T,\omega) &= J(X_T,\omega),\\
\tilde Q_t(X_t,U^b_t,\omega) &= \inf_{U^a_{t+1}\in\reals^{M^a}}\{L_t(X_t,U^b_t,U^a_{t+1},\omega)  + J_{t+1}(F_t(X_t,U^b_t,U^a_{t+1},\omega),\omega)\},\\
J_t(X_t,\omega) &= \inf_{U^b_t \in\reals^{M^b}}(E_t\tilde Q_t)(X_t,U^b_t,\omega)
\end{split}
\end{equation}
for every $t=0,\ldots,T$ and $(X,U^b,U^a)\in(\reals^{N+M^b+M^a})^{T+1}$. We then say that the sequence $(J_t,\tilde Q_t)_{t=0}^T$ is a {\em normal solution} of \eqref{be}.

If the conditional expectations $E_t$ can be expressed as ``regular conditional expectations'', then the last equation in \eqref{be} can be expressed pointwise using the expression \eqref{eq:kernel}; see also \cite[Example~2.102]{pp24}. Substituting out the normal integrand $\tilde Q_t$, then shows that the normal integrands $J_t$ in \eqref{be} satisfy \eqref{bei} and, in particular, that the infimums in \eqref{bei} are measurable functions. Indeed, the pointwise infimum of a normal integrand is automatically measurable; see e.g.\ \cite[Corollary~1.23]{pp24}.

The three theorems below will be proved in Section~\ref{sec:oc}. The first one characterizes the optimum value and optimal solutions of \eqref{oc} in terms of the solutions of the Bellman equations \eqref{be}. In continuous time control theory, such theorems are sometimes called ``verification theorems''.  

Given a normal integrand $h:\reals^n\times\Omega\to\ereals$, the function $h^\infty$, obtained from $h$ by defining $h^\infty(\cdot,\omega)$, for each $\omega\in\Omega$, as the {\em recession function}\footnote{Given a lower semicontinuous extended real-valued convex function $g$ on $\reals^n$, its {\em recession function} is defined by
\[
g^\infty(x):=\lim_{\alpha\to\infty}\frac{g(\bar x+\alpha x)-g(\bar x)}{\alpha},
\]
where $\bar x\in\dom g$; see \cite[Theorem~8.5]{roc70a}.}
of $h(\cdot,\omega)$, is a positively homogeneous convex normal integrand; see e.g.~\cite[Example~14.54a]{pp24} or \cite[Theorem~1.39]{pp24}. We will assume that
\begin{equation}\label{eq:ass}
\{U_t^a \in \reals^{M^a}\mid L^\infty_{t-1}(0,0,U^a_t,\omega) \le 0,\ B_t^a(\omega) U^a_t = 0\}
\end{equation}
is linear for every $t=1,\dots,T$ and almost every $\omega$. This condition holds automatically under the conditions of \thref{existoc}; see \thref{rem:existoc}. We will use the notation $\N^s:=\{(x_t)_{t=0}^s\mid x\in\N\}$.

\begin{theorem}[Optimality principle]\thlabel{opoc}
If $(J_t,\tilde Q_t)_{t=0}^T$ is a lower bounded convex normal solution of~\eqref{be}, then the optimum value of~\eqref{oc} equals that of
\begin{align*}
  &\minimize & E\Bigg[\sum_{t=0}^{s-1}& L_t(X_t,U^b_t,U^a_{t+1})+ J_s(X_s)\Bigg]\ \ovr\ (X,U^b,U^a)\in\N^s,\\ 
  &\st\ & X_{t+1} &= F_t(X_t,U^b_t,U^a_{t+1})\quad t=0,\dots,s-1\ a.s.,
\end{align*}
for every $s=0,\ldots,T$, and an $(\bar X,\bar U^b,\bar U^a)\in\N$ solves~\eqref{oc} if and only if
\begin{align*}
  X_0&\in\argmin_{x_0\in\reals^N}J_0(x_0),\\
  \bar U^b_t&\in \argmin_{U^b_t\in\reals^{M^b}} (E_t\tilde Q_t)(X_t,U^b_t),\\
  \bar U^a_{t+1} &\in \argmin_{u^a_{t+1}\in\reals^{M^a}}\{L_t(\bar X_t,\bar U^b_t,u^a_{t+1}) + J_{t+1}(F_t(\bar X_t,\bar U^b_t,u^a_{t+1}))\},\\
  \bar X_{t+1} &= F_{t}(\bar X_t,\bar U^b_t,\bar U^a_{t+1})
\end{align*}
for all $t=0,\ldots,T-1$ almost surely.
\end{theorem}
The optimality condition in \thref{opoc} characterizes the optimal solutions of \eqref{oc}, for each time $t$ and state $\omega$ as solutions of a two-stage convex optimization problems. This is a major simplification over \eqref{oc} where one optimizes over adapted processes $(X,U^b,U^a)$ over $T$ stages. The above, of course, requires knowledge of the Bellman functions $J_t$ which are normal integrands on $\reals^N\times\Omega$. While in \eqref{oc}, one optimizes at each stage over $(X_t,U^b_t,U^a_t)$, the Bellman functions $J_t$ only depend on the system state $X_t$. When the problem does not depend on $U^a$, \thref{opoc} reduces to \cite[Theorem~2.91]{pp24}, where the optimality condition concerning $X_0$ was omitted by mistake.

The following gives sufficient conditions for the existence of solutions to the Bellman equations \eqref{be}. 

\begin{theorem}[Existence of solutions]\thlabel{existoc}
Assume that the set
\begin{multline*}
\{(X,U^b,U^a)\in\N\mid \sum_{t=0}^{T-1} L^\infty_t(X_t,U^b_t,U^a_{t+1})+J^\infty(X_T)\le 0,\\
X_t=A_tX_{t-1} + B^b_tU^b_{t-1} + B^a_tU^a_t\ t=1,\dots,T\ a.s.\}
\end{multline*}
is linear. Then \eqref{oc} has a solution $(\bar X,\bar U^b,\bar U^a)\in\N$ and the Bellman equations \eqref{be} have a unique lower bounded convex normal solution $(J_t,\tilde Q_t)_{t=0}^T$.
\end{theorem}

The linearity assumption in \thref{existoc} holds in particular if the lower level sets of the cost functions $L_t$ and $J$ are scenariowise bounded. In clasical models of financial mathematics, the linearity condition becomes the ``no-arbitrage condition'' which does not require boundedness; see \cite[Section~2.3.5]{pp24}.

\begin{remark}\thlabel{rem:existoc}
The condition of \thref{existoc} implies that the set
\begin{align*}
\{U^a_{t+1}\in L^0(\F_{t+1})\mid L^\infty_t(0,0,U^a_{t+1})\le 0,\ B^a_{t+1}U^a_{t+1} =0\ a.s.\}
\end{align*}
is linear for every $t$. By \cite[Corollary~1.74]{pp24}, this holds if and only if the set in \eqref{eq:ass} is linear for almost every $\omega$.
\end{remark}

So far, we have not assumed anything about the underlying randomness in \eqref{oc}. Consequently, the cost-to-go functions $J_t$ are merely $\F_t$-measurable normal integrands. The independence assumption made in \cite[Section~6]{ccdmt}, on the other hand, resulted in nonrandom functions cost-to-go functions. We can reduce the dimensionality of the Bellman equations with respect to randomness more generally as follows.

Given an $\reals^d$-valued random variable $\xi$, we say that another random variable $C$ depends on $\omega$ only through $\xi$ if there is a measurable function $\hat C$ on $\reals^d$ such that $C(\omega)=\hat C(\xi(\omega))$ almost surely. By the Doob-Dynkin lemma, this happens if and only if $C$ is $\sigma(\xi)$-measurable. Similarly, we say that a normal integrand $h:\reals^n\times\Omega\to\ereals$ depends on $\omega$ only through $\xi$ if there is a normal integrand $H:\reals^n\times\reals^d\to\ereals$ such that $h(x,\omega)=H(x,\xi(\omega))$ for all $x\in\reals^n$ almost surely. Here, we endow $\reals^d$ with the Borel $\sigma$-algebra $\B(\reals^d)$. By \cite[Corollary~1.34]{pp24} such a representation exists if and only if $h$ is a $\sigma(\xi)$-measurable normal integrand.

\begin{theorem}[Dimensionality with respect to scenarios]\thlabel{dimxioc}
Let $(J_t,\tilde Q_t)_{t=0}^T$ be a lower bounded convex normal solution of \eqref{be} and assume that $L_{t-1}$, $A_t$, $B^b_t$, $B^a_t$ and $W_t$ depend on $\omega$ only through random variables $(\theta_{t-1},\theta_t,\eta_t)$, where $\theta=(\theta)_{t=0}^T$ is an adapted Markov process and $\eta=(\eta_t)_{t=0}^T$ is an adapted sequence of random variables such that each $\eta_t$ is independent of $\theta$ and of $\eta_s$ for $s\ne t$. If $J$ depends on $\omega$ only through $\theta_T$, then $J_t$ depends on $\omega$ only through $\theta_t$. In particular, if $L_{t-1}$, $A_t$, $B^b_t$, $B^a_t$, $W_t$ and $J$ are independent of $\theta$, then $J_t$ is deterministic.
\end{theorem}

\begin{remark}\label{rem:hd}
Defining $Q_t:=E_t\tilde Q_t$, we can write \eqref{be} in hazard-decision format as
\begin{equation*}
\begin{split}
Q_T(X_T,U^b_T) &= J(X_T),\\
\tilde Q_t(X_t,U^b_t) &= \inf_{(U^a_{t+1},\tilde U_{t+1})\in\reals^{M^a+M^b}}\{L_t(X_t,U^b_t,U^a_{t+1})  + Q_{t+1}(F_t(X_t,U^b_t,U^a_{t+1}),\tilde U_{t+1})\},\\
Q_t&=E_t\tilde Q_t.
\end{split}
\end{equation*}
This can be interpreted as the Bellman equations for the hazard-decision problem with state $(X_t,U^b_t)$, control $(U^a_t,\tilde U_t)$ and system equations
\[
\begin{bmatrix}X_{t+1}\\
  U^b_{t+1}
\end{bmatrix}
= \begin{bmatrix}
  F_t(X_t,U^b_t,U^a_{t+1})\\
  \tilde U_{t+1}
  \end{bmatrix}.
\]
Such a reformulation of \eqref{oc} corresponds to the hazard-decision reformulation suggested in \cite{svlv} and \cite{dow20} in the context of the Stochastic Dual Dynamic Programming (SDDP) algorithm. One should note that the dimensionality of the Bellman equations in the hazard-decision reformulation above is $N+M^b$ while in the decision-hazard-decision formulation it is $N$. The dimension of the state may have a significant effect on the performance of e.g.\ the SDDP algorithm which is based on Kelley's cutting plane algorithm whose performance deteriorates quickly with problem dimension; see \cite[Section~3.3.2]{nes18}. This motivates the Bellman equations in the decision-hazard-decision format \eqref{be}
where the Bellman functions $J_t$ depend only on the $N$-dimensional system state $X_t$. Section~\ref{sec:sddp} below outlines a version of the SDDP algorithm applicable to \eqref{be}.
\end{remark}

\section{Dynamic programming in problems of Lagrange}\label{sec:L}

This section studies an abstract stochastic optimization model inspired by calculus of variations; see \cite{rw83} or \cite{pp24} and their references. This format does not split the decision variables to states and controls as in \eqref{oc} so the cost-to-go function at time $t$ will be a function of all the decision variables chosen at time $t$. The format covers many more specific formats studied in stochastic programming and, in particular, in the context of the stochastic dual dynamic programming algorithm; see e.g.~\cite{pp91}.

Consider the problem
\begin{equation}\label{L}
\minimize\quad E\left[\sum_{t=1}^T k_t(x_{t-1},x_t) + Q(x_T)\right]\quad\ovr x\in\N.
\end{equation}
where $k_t:\reals^d\times\reals^d\times\Omega\to\ereals$ are $\F_t$-measurable and $Q_T:\reals^d\times\Omega\to\ereals$ is $\F_T$-measurable lower bounded convex normal integrands. It is clear that the terminal cost $Q$ is redundant in the sense that one could simply redefine $k_T$ by adding $Q$ to it. We have included it mainly for notational convenience in the applications of the following sections. When the functions $k_t$ are of the form
\begin{equation}\label{eq:fl}
  k_t(x_{t-1},x_t,\omega) = \begin{cases}
  l_t(x_t,\omega) & \text{if $(x_t,x_{t-1})\in C_t(\omega)$},\\
    +\infty & \text{otherwise},
\end{cases}
\end{equation}
we recover the model studied in \cite{fl23}. Here $l_t$ is an $\F_t$-measurable normal integrand and $C_t$ is an $\F_t$-measurable closed convex set. On the other hand, if
\begin{equation}\label{eq:ls}
k_t(x_{t-1},x_t,\omega) = \begin{cases}
  f_t(x_t,\omega) & \text{if $x_t\in D_t(\omega)$ and $B_t(\omega)x_t+A_t(\omega)x_{t-1}=b_t(\omega)$},\\
  +\infty & \text{otherwise},
\end{cases}
\end{equation}
we recover the model studied e.g.\ in \cite[Chapter~3]{sdr21}. Here $f_t$ is an $\F_t$-measurable convex normal integrand, $D_t$ is an $\F_t$-measurable closed convex set and $A_t$, $B_t$, $b_t$ are $\F_t$-measurable matrices of an appropriate dimensions. Indeed, by \cite[Theorem~1.20]{pp24}, the function given by \eqref{eq:ls} is a convex normal integrand. More examples will be given in the following sections. In particular, we will see in Section~\ref{sec:oc} that the stochastic control model in the introduction is a special case of problem~\eqref{L}.

Consider the following Bellman equations
\begin{equation}\label{bel}
  \begin{split}
    Q_T&=Q,\\
    \tilde Q_t(x_t) &= \inf_{x_{t+1}\in\reals^d}\{k_{t+1}(x_t,x_{t+1}) + Q_{t+1}(x_{t+1})\},\\
    Q_t &= E_t\tilde Q_t.
  \end{split}
\end{equation}
The following characterizes the optimum value and optimal solutions of \eqref{L} in terms of the solutions $(Q_t)_{t=0}^T$ of \eqref{bel}. Recall that $\N^s:=\{(x_t)_{t=0}^s\mid x\in\N\}$.

\begin{theorem}[Optimality principle]\thlabel{opl}
Let $(Q_t,\tilde Q_t)_{t=0}^T$ be a lower bounded convex solution of \eqref{bel}. Then the optimum value of \eqref{L} equals that of
\[
\minimize\quad E\left[\sum_{t=1}^sk_t(x_{t-1},x_t) + Q_s(x_s)\right]\quad\ovr x\in\N^s
\]
for all $s=0,\dots,T$ and, moreover, an $\bar x\in\N$ solves \eqref{L} if and only if
\begin{align*}
\bar x_0&\in\argmin_{x_0\in\reals^d}Q_0(x_0),\\
\bar x_t&\in\argmin_{x_t\in\reals^d}\{k_t(\bar x_{t-1},x_t) + Q_t(x_t)\}\quad t=1,\ldots,T,
\end{align*}
almost surely.
\end{theorem}

\begin{proof}
We apply \cite[Theorem~2.106]{pp24} with
\begin{align*}
  K_0(x_0,u_0,\omega):=& 0,\\
  K_t(x_t,u_t,\omega) :=& k_t(x_t-u_t,x_t,\omega)\quad t=1,\ldots,T-1,\\
  K_T(x_t,u_t,\omega) :=& k_T(x_T-u_T,x_T,\omega)+Q(x_T,\omega).
\end{align*}
By \cite[Theorem~1.20]{pp24}, the functions $K_t$ are normal integrands. It is clear that the functions $K_t$ are also lower boundedness and convex since we have assumed that $k_t$ are so. The claims thus follow from \cite[Theorem~2.106]{pp24}.
\end{proof}

Choosing $s=0$, the first part of \thref{opl} combined with the interchange rule in \cite[Theorem~14.60]{rw98} imply that the optimum value of \eqref{L} equals
\[
\inf\{E[Q_0(x_0)]\mid x_0\in L^0(\Omega,\F_0,P;\reals^{n_0})\} = E\inf_{x_0\in\reals^{n_0}}Q_0(x_0).
\]

\begin{theorem}[Existence of solutions]\thlabel{existl}
Assume that the set
\[
\{x\in\N\mid \sum_{t=1}^T k^\infty_t(x_{t-1},x_t)+Q^\infty(x_T)\le 0\ a.s.\}
\]
is linear. Then \eqref{L} has a solution $\bar x\in\N$, the Bellman equations \eqref{bel} have a unique lower bounded convex normal solution $(Q_t,\tilde Q_t)_{t=0}^T$ and the sets
\begin{equation}\label{eq:Lagrec}
N_t(\omega):=\{x_t\in\reals^d \mid k_t^\infty(0,x_t,\omega) + Q_t^\infty(x_t,\omega)\le 0\}
\end{equation}
are linear for every $t$ and almost every $\omega$.
\end{theorem}

\begin{proof}
We apply  \cite[Theorem~2.108]{pp24} with 
\begin{align*}
  K_0(x_0,u_0,\omega):=& 0,\\
  K_t(x_t,u_t,\omega) :=& k_t(x_t-u_t,x_t,\omega)\quad t=1,\ldots,T-1,\\
  K_T(x_t,u_t,\omega) :=& k_T(x_T-u_T,x_T,\omega)+Q(x_T,\omega).
\end{align*}
By \cite[Theorem~A.17]{pp24},
\begin{align*}
  K^\infty_0(x_0,u_0,\omega):=& 0,\\
  K^\infty _t(x_t,u_t,\omega) :=& k^\infty_t(x_t-u_t,x_t,\omega)\quad t=1,\ldots,T-1,\\
  K_T(x_t,u_t,\omega) :=& k^\infty_T(x_T-u_T,x_T,\omega)+Q^\infty(x_T,\omega),
\end{align*}
so the linearity condition here implies that in \cite[Theorem~2.108]{pp24}. Thus, the second claim follows from \cite[Theorem~2.108]{pp24}. Combining the last statements of  \cite[Theorem~2.108]{pp24} and \cite[Theorem~2.106]{pp24} proves the remaining claim. 
\end{proof}

\begin{theorem}[Dimensionality with respect to scenarios]\label{dimxi}
Let $(Q_t,\tilde Q_t)_{t=0}^T$ be a lower bounded convex normal solution of the Bellman equations \eqref{bel} and assume that $k_t$ depends on $\omega$ only through random variables $(\theta_{t-1},\theta_t,\eta_t)$, where $\theta=(\theta)_{t=0}^T$ is an adapted Markov process and $\eta=(\eta_t)_{t=0}^T$ is an adapted sequence of random variables such that each $\eta_t$ is independent of $\theta$ and of $\eta_s$ for $s\ne t$. If $Q$ depends on $\omega$ only through $\theta_T$, then $Q_t$ depends on $\omega$ only through $\theta_t$. In particular, if $k_t$ and $Q$ are independent of $\theta$, then $Q_t$ is deterministic.
\end{theorem}

\begin{proof}
The claim clearly holds for $t=T$. Assume that it holds for $t+1$. The normal integrand $\tilde Q_t$ then depends on $\omega$ only through $(\theta_t,\theta_{t+1},\eta_{t+1})$. Applying \cite[Theorem~2.21]{pp24} with $\G:=\F_t$ and $\H:=\sigma(\theta_t)$ implies that $E[\tilde Q_t\mid\F_t]=E[\tilde Q_t\mid\sigma(\theta_t)]$. In particular, $Q_t$ is $\sigma(\theta_t)$-measurable so the claim follows from \cite[Corollary~1.34]{pp24}.
\end{proof}

\section{Dimensionality of the Bellman equations}\label{sec:dim}

An essential feature of the optimal control problem \eqref{oc} is that the control variables at a given stage affect the following stages only through the state variable $X_t$. Accordingly, the Bellman functions $J_t$ in \eqref{be} only depend on the state instead of all the decision variables of stage $t$ as in \eqref{bel}. Such a dimension reduction can be beneficial when numerically constructing approximations of the Bellman functions by e.g.\ cutting plane methods; see \cite{nes18}. The structure behind the dimension reduction generalizes as follows.

Consider again problem \eqref{L} and assume that the decision variables decompose as $x_t=(X_t,U_t)$ and that
\begin{align*}
  k_t(x_{t-1},x_t,\omega) &= K_t(X_{t-1},U_{t-1},X_t,\omega),\\
  Q(x_T) &= J(X_T),
\end{align*}
where $K_t$ is a lower bounded convex normal integrand on $\reals^N\times\reals^{M}\times\reals^N\times\Omega$ and $J$ is a lower bounded convex normal integrand on $\reals^N\times\Omega$. We can then write problem \eqref{L} as
\begin{equation}\label{L2}
\begin{aligned}
\minimize\quad E\left[\sum_{t=1}^T K_t(X_{t-1},U_{t-1},X_t) + J(X_T)\right]\quad\ovr (X,U)\in\N
\end{aligned}
\end{equation}
and the corresponding Bellman equations \eqref{bel} as
\begin{equation}\label{eq:beLdim}
\begin{aligned}
  Q_T(X_T,U_T) &= J(X_T),\\
  \tilde Q_t(X_t,U_t) &= \inf_{(X_{t+1},U_{t+1})\in\reals^{N+M}}\{K_{t+1}(X_t,U_t,X_{t+1}) + Q_{t+1}(X_{t+1},U_{t+1})\},\\
  Q_t &= E_t\tilde Q_t.
\end{aligned}
\end{equation}
If $(Q_t,\tilde Q_t)_{t=0}^T$ satisfy \eqref{eq:beLdim}, then the functions
\[
J_t(X_t) := \inf_{U_t\in\reals^M} Q_t(X_t,U_t)
\]
together with $\tilde Q_t$ satisfy the reduced Bellman equations
\begin{equation}\label{ber}
\begin{aligned}
J_T(X_T) &= J(X_T),\\
\tilde Q_t(X_t,U_t) &= \inf_{X_{t+1}\in\reals^N}\{K_{t+1}(X_t,U_t,X_{t+1}) + J_{t+1}(X_{t+1})\},\\
J_t(X_t) &= \inf_{U_t\in\reals^M}(E_t\tilde Q_t)(X_t,U_t).
\end{aligned}
\end{equation}
As in the control format of \eqref{be}, the Bellman functions $J_t$ above only depend on the state $X_t$, not on the control $U_t$. Section~\ref{sec:sddp} presents a version of the Stochastic Dual Dynamic Programming algorithm applicable to the above format. A version of the SDDP algorithm for problem \eqref{oc} is obtained as a special case. 


\begin{theorem}[Optimality principle]\thlabel{opd}
Let $(J_t,\tilde Q_t)_{t=0}^T$ be a lower bounded convex solution of \eqref{ber}. Then the optimum value of \eqref{L2} equals that of
\begin{align*}
\minimize\quad E\left[\sum_{t=1}^s K_t(X_{t-1},U_{t-1},X_t) + J_s(X_s)\right]\quad (X,U)\in\N^s
\end{align*}
for all $s=0,\dots,T$ and, moreover, an $(\bar X,\bar{U})\in\N$ solves \eqref{L2} if and only if
\begin{align*}
  \bar X_0 &\in\argmin_{X_0\in\reals^N}J_0(X_0),\\
  \bar{U}_t &\in\argmin_{U_t\in\reals^M}(E_t\tilde Q_t)(\bar X_t,U_t),\\
  \bar X_{t+1} &\in \argmin_{X_{t+1}\in\reals^N}\{K_{t+1}(\bar X_t,\bar U_t,X_{t+1}) + J_{t+1}(X_{t+1})\}
\end{align*}
for all $t=0,\ldots,T$ almost surely.
\end{theorem}

\begin{proof}
The functions $Q_t:=E_t[\tilde Q_t]$ satisfy \eqref{eq:beLdim} so the first claim of \thref{opl} says that the optimum value of \eqref{L2} equals
\begin{align*}
  \inf_{x\in\N^s}E\left[\sum_{t=1}^sk_t(x_{t-1},x_t) + Q_s(x_s)\right] &= \inf_{(X,U)\in\N^s}E\left[\sum_{t=1}^s K_t(X_{t-1},U_{t-1},X_t) + Q_s(X_s,U_s)\right]\\
  &= \inf_{(X,U)\in\N^s}E\left[\sum_{t=1}^s K_t(X_{t-1},U_{t-1},X_t) + J_s(X_s)\right]
\end{align*}
where the second equality follows from the interchange rule in \cite[Theorem~14.60]{rw98}. Clearly, $(\bar X_0,\bar U_0)$ minimizes $Q_0$ if and only if $\bar X_0$ minimizes $J_0$ and $\bar U_0$ minimizes $Q_0(\bar X_0,\cdot)$. Similarly, $(\bar X_t,\bar U_t)$ minimizes the function
\[
f_t(X_t,U_t):=k_t(\bar x_{t-1},x_t) + Q_t(x_t) = K_t(\bar X_{t-1},\bar U_{t-1},X_t) + Q_t(X_t,U_t)
\]
if and only if $\bar X_t$ minimizes the function
\[
\varphi_t(X_t):=\inf_{U_t\in\reals^M}f_t(X_t,U_t) = K_t(\bar X_{t-1},\bar U_{t-1},X_t) + J_t(X_t)
\]
and $\bar U_t$ minimizes $Q_t(\bar X_t,\cdot)$. The second claim thus follows from that of \thref{opl}.
\end{proof}

As in the general Lagrangian format of Section~\ref{sec:L}, the first claim of \thref{opd} implies that the optimum value of \eqref{L2} equals
\[
\inf_{X_0\in\reals^N}E[J_0(X_0)].
\]

\begin{theorem}[Existence of solutions]\thlabel{existd}
Assume that the set
\[
\{(X,U)\in\N\mid \sum_{t=0}^T K^\infty _t(X_{t-1},U_{t-1},X_t) + J^\infty(X_T)\le 0\ a.s.\}
\]
is linear. Then \eqref{L2} has a solution $(\bar X,\bar U)\in\N$ and the Bellman equations \eqref{ber} have a unique lower bounded convex normal solution $(\tilde Q_t,J_t)_{t=0}^T$.
\end{theorem}

\begin{proof}
By \thref{existl}, \eqref{L2} has a solution and the Bellman equations \eqref{eq:beLdim} have a unique solution $(Q_t,\tilde Q_t)_{t=0}^T$ of lower bounded convex normal integrands.  The linearity of  \eqref{eq:Lagrec} implies the linearity of the set
\[
\{U_t\in\reals^M \mid Q_t^\infty(0,U_t,\omega)\le 0\}.
\]
Thus, by \cite[Theorem~1.40]{pp24}, $J_t$ are normal integrands.
\end{proof}

The following is a direct consequence of \autoref{dimxi}.

\begin{theorem}[Dimensionality with respect to scenarios]\thlabel{dimxidim}
Let $(J_t,\tilde Q_t)_{t=0}^T$ be a lower bounded convex normal solution of the Bellman equations \eqref{ber} and assume that $K_t$ depends on $\omega$ only through random variables $(\theta_{t-1},\theta_t,\eta_t)$, where $\theta=(\theta)_{t=0}^T$ is an adapted Markov process and $\eta=(\eta_t)_{t=0}^T$ is an adapted sequence of random variables such that each $\eta_t$ is independent of $\theta$ and of $\eta_s$ for $s\ne t$. If $J$ depends on $\omega$ only through $\theta_T$, then $J_t$ depends on $\omega$ only through $\theta_t$. In particular, if $K_t$ and $J$ are independent of $\theta$, then $J_t$ is deterministic.
\end{theorem}

\section{Proofs of Theorems~\ref{opoc}, \ref{existoc} and \ref{dimxioc}}\label{sec:oc}

The optimal control problem in the introduction of this paper is an instance of problem \eqref{L2} with $U_t=U^b_t$ and
\begin{multline}\label{eq:ock}
K_t(X_{t-1},U_{t-1}, X_t,\omega) := \\
 \inf_{U_t^a\in\reals^{M^a}}\{L_{t-1}(X_{t-1},U_{t-1},U^a_t,\omega) \mid X_t = F_{t-1}(X_{t-1},U_{t-1},U^a_t,\omega)\}.
\end{multline}
By \cite[Theorem~1.40]{pp24}, the linearity of \eqref{eq:ass} implies that $K_t$ is a convex normal integrand. The idea of treating optimal control problems as special cases of problems of Lagrange goes back to \cite[Example~3]{roc70b}; see also \cite[Example~II.3]{bis73b}.

With \eqref{eq:ock}, the Bellman equations \eqref{ber} can be written as
\[
\begin{split}
J_T &= J,\\
\tilde Q_t(X_t,U^b_t) &= \inf_{U^a_{t+1}\in\reals^{M^a}}\{L_t(X_t,U^b_t,U^a_{t+1})  + J_{t+1}( F_t(X_t,U^b_t,U^a_{t+1}))\},\\
J_t(X_t) &= \inf_{U^b_t\in\reals^{M^b}}(E_t\tilde Q_t)(X_t,U^b_t),
\end{split}
\]
which are the Bellman equations \eqref{be} for the optimal control problem \eqref{oc}.

\begin{proof}[Proof of \thref{opoc}]
With \eqref{eq:ock}, the optimum values of \eqref{L2} and \eqref{oc} coincide and $(\bar X,\bar U^b,\bar U^a)\in\N$ solves \eqref{oc} if and only if $(\bar X,\bar U^b)$ solves \eqref{L2} and
\begin{equation}\label{eq:ocl}
\bar U^a_t\in\argmin_{U_t^a\in\reals^{M^a}}\{L_{t-1}(\bar X_{t-1},\bar U^b_{t-1},U^a_t) \mid \bar X_t = F_{t-1}(\bar X_{t-1},\bar U^b_{t-1},U^a_t)\}
\end{equation}
for all $t=0,\ldots,T$ for almost every $\omega$. Indeed, by the interchange rule in \cite[Theorem~14.60]{rw98},
\begin{multline*}
  \inf_{U^a_t\in L^0(\Omega,\F_t,P;\reals^{M^a})}\{E[L_{t-1}(\bar X_{t-1},\bar U^b_{t-1},U^a_t] \mid X_t = F_{t-1}(\bar X_{t-1},\bar U_{t-1},U^a_t)\}\\
  = E[K_t(\bar X_{t-1},\bar U^b_{t-1},\bar X_t)]
\end{multline*}
and an $\bar U^a_t$ attains the above infimum if and only if \eqref{eq:ocl} holds almost surely. Similarly, the first claim of \thref{opd} implies that of \thref{opoc}.

By \cite[Theorem~1.40]{pp24}, the linearity of \eqref{eq:ass} implies that the infimum in  \eqref{eq:ock} is attained, so
\begin{align*}
 \bar X_{t+1} &\in \argmin_{X_{t+1}\in\reals^N}\{K_{t+1}(\bar X_t,\bar U_t,X_{t+1}) + J_{t+1}(X_{t+1})\}
\end{align*}
if and only if
\begin{align*}
  \bar U^a_{t+1} &\in \argmin_{U^a_{t+1}\in\reals^{M^a}}\{L_t(\bar X_t,\bar U^b_t,u^a_{t+1}) + J_{t+1}(F_t(\bar X_t,\bar U^b_t,U^a_{t+1}))\},\\
  \bar X_{t+1} &= F_{t}(\bar X_t,\bar U^b_t,\bar U^a_{t+1}).
\end{align*}
Thus, the optimality conditions in \thref{opd} can be written as those in \thref{opoc}.
\end{proof}

\begin{proof}[Proof of \thref{existoc}]
By \cite[Theorem~A.17, Theorem~1.40]{pp24}, the linearity of \eqref{eq:ass} implies that, for every $(X,U)\in\N$, there exists an adapted $U^a$ such that
\[
K^\infty_t(X_{t-1},U_{t-1}, X_t) := L^\infty _{t-1}(X_{t-1},U_{t-1},U^a_t)
\]
with $X_t = A_tX_{t-1}+B^b_tU^b_{t-1}+B^a_t U^a_t$ for all $t$ almost surely. Thus, 
\begin{multline*}
  \{(X,U)\in\N\mid \sum_{t=0}^T K^\infty _t(X_{t-1},U_{t-1},X_t) + J^\infty(X_T)\le 0\ a.s.\}\\
  =\{(X,U)\in\N\mid \exists U^a\in\N:\ \sum_{t=0}^T L_{t-1}^\infty(X_{t-1},U^b_{t-1},U^a_t) + J^\infty(X_T)\le 0,\\
  X_t = A_tX_{t-1} +B^a_t U^b_{t-1}+B^b_t U^a_t\ a.s.\}.
\end{multline*}
The linearity condition in \thref{existoc} thus implies the linearity condition in \thref{existd}. This gives the existence of solutions to \eqref{ber}. As noted above, this gives the existence of solutions to \eqref{be}. \thref{existd} also gives the existence of solutions to \eqref{L2}. As observed in the proof of \thref{opoc}, this together with the linearity of \eqref{eq:ass} gives the existence of solutions to \eqref{oc}.
\end{proof}

\begin{proof}[Proof of \thref{dimxioc}]
The dependence assumptions in \thref{dimxioc} imply those in \thref{dimxidim}. Thus, the claims follow from \thref{dimxidim}.
\end{proof}

\section{Stochastic dual dynamic programming}\label{sec:sddp}

Stochastic dual dynamic programming (SDDP) is an iterative algorithm for constructing approximate solutions to Bellman equations. SDDP was proposed in \cite{pp91} under the assumption that the randomness is driven by independent noises (no Markov process $\theta$) so that the Bellman functions are deterministic; see \autoref{dimxioc}. Moreover, \cite{pp91} did not split the decision variables into states and controls so the cost-to-go function at time $t$ was a function of all the decision variables chosen at time~$t$. The SDDP algorithm was extended to linear problems with non-independent noises in \cite{im96}; see also \cite{pm12,ls19} and their references.

This section gives a further extension of the SDDP algorithm that applies to the general decision-hazard-decision format of Section~\ref{sec:dim} that allows for nonlinear constraints and objectives, Markovian randomness and cost-to-go functions that only depend on the linking variables $X_t$.

Assume that $K_t$ are as in Section~\ref{sec:dim} so that the Bellman functions $J_t$ depend on the decision variables only through the system state $X_t$. We will also assume, as in \autoref{dimxidim}, that $K_t$ depends on $\omega$ only through random variables $(\theta_{t-1},\theta_t,\eta_t)$, where $\theta=(\theta)_{t=0}^T$ is an adapted Markov process and $\eta=(\eta_t)_{t=0}^T$ is an adapted sequence of random variables such that each $\eta_t$ is independent of $\xi$ and of $\eta_s$ for $s\ne t$. Similarly, we assume that $J$ depends on $\omega$ only through $\theta_T$. By \autoref{dimxidim}, each Bellman function $J_t$ then depends on $\omega$ only through $\theta_t$. Furthermore, we will now assume that $\xi:=(\theta,\eta)$ is finitely supported so that the conditional expectations below become finite sums.

For problems of the form \eqref{L2}, the SDDP algorithm proceeds as follows:

{\footnotesize
\begin{enumerate}
\item[0.]
{\bf Initialization:} Choose lower-approximations $(X_t,\theta_t)\mapsto J_t^0(X_t,\theta_t)$ of the cost-to-go functions $J_t$ such that $J_t^0(\cdot,\theta_t)$ are convex (polyhedral) and $J_T^0=J$. Set $k=0$.
\item
{\bf Forward pass:} Sample a path $\xi^k$ of $\xi$ and define $X^k_t$ for $t=0,\ldots,T$ by
\begin{align*}
  X_0^k&\in\argmin_{X_0\in\reals^N} J^k_0(X_0,\theta_0^k),\\
  U^k_t&\in\argmin_{U_t\in\reals^M}E\left[\inf_{X_{t+1}\in\reals^N}\{K_{t+1}(X^k_t,U_t,X_{t+1}) + J^k_{t+1}(X_{t+1})\}\Bigg{|}\theta^k_t\right],\\
  X^k_{t+1} &\in\argmin_{X_{t+1}\in\reals^N}\{K_{t+1}(X^k_t,U^k_t,X_{t+1},\xi^k) + J^k_{t+1}(X_{t+1},\theta_{t+1}^k)\}.
\end{align*}
\item
{\bf Backward pass:} Let $J_T^{k+1}:=J$ and, for $t=T-1,\ldots,0$,
\begin{align*}
  \tilde J_t^{k+1}(X_t^k,\theta_t^k) &:= \inf_{U_t\in\reals^M}E\left[\inf_{X_{t+1}\in\reals^N}\{K_{t+1}(X^k_t,U_t,X_{t+1}) + J^{k+1}_{t+1}(X_{t+1})\}\Bigg{|}\theta^k_t\right],\\
V^{k+1}_t&\in\partial\tilde J^{k+1}_t(X_t^k,\theta_t^k),
\end{align*}
and let
\[
J_t^{k+1}(X_t,\theta_t^k) := \max\{J_t^k(X_t,\theta_t^k),\tilde J_t^{k+1}(X_t^k,\theta_t^k) + V_t^{k+1}\cdot(X_t-X^k_t)\}\quad\forall X_t\in\reals^N.
\]
Set $k:=k+1$ and go to 1.
\end{enumerate}
}

The SDDP algorithm above, decomposes the original $(T+1)$-stage optimization problem \eqref{L2} into a series of two-stage convex stochastic optimization problems. If the $\theta_t^k$-conditional distribution of $\xi_{t+1}$ is supported by $m_{t+1}$ scenarios, then the second stage has $m_{t+1}$ copies of the variables $(X_{t+1},U^a_{t+1})$. The optimal solution $(X^k_{t+1},U^{ak}_{t+1})$ in the forward pass is the optimal second stage solution corresponding to the scenario $\xi^k$. In the hazard-decision format where the ``here-and-now'' variables $U^b_t$ are absent, the second stage optimization can be done separately for each scenario and, in the forward pass, only the scenario $k$ needs to be treated.

The subgradients $V_t^{k+1}$ of $\tilde J_t^{k+1}$ at $X_t^k$ in the backward pass are provided by standard optimization packages. Even in the abstract setting above, one can always introduce a dummy variable $\tilde X_t$ and add a constraint $\tilde X_t=X_t^k$ and then, the Lagrange multipliers of this constraint are subgradients of $\tilde J_t^{k+1}$ at $X_t^k$.

In the optimal control format of the introduction and Section~\ref{sec:oc}, one can use the system equations to substitute out the state variables $X_{t+1}$ so the SDDP algorithm above can be written as follows:

{\footnotesize
  \begin{enumerate}
  \item[0.]
    {\bf Initialization:} Choose lower-approximations $(X_t,\theta_t)\mapsto J_t^0(X_t,\theta_t)$ of the cost-to-go functions $J_t$ such that $J_t^0(\cdot,\theta_t)$ are convex (polyhedral) and  $J_T^0=J$. Set $k=0$.
  \item
    {\bf Forward pass:} Sample a path $\xi^k$ of $\xi$ and define $X^k_t$ for $t=0,\ldots,T$ by
    \begin{align*}
      X_0^k&\in\argmin_{X_0\in\reals^N} J^k_0(X_0,\theta_0^k),\\
      U_t^{bk}&\in\argmin_{U^b_t\in\reals^{M^b}} E\left[\inf_{U^a_{t+1}\in\reals^{M^a}}\{L_t(X^k_t,U^b_t,U^a_{t+1}) + J^k_{t+1}(F_{t}(X^k_t,U^b_t,U^a_{t+1}))\}
        \Bigg{|}\theta^k_t\right],\\
      U_{t+1}^{ak} &\in\argmin_{U^a_{t+1}\in\reals^{M^a}}\{L_t(X^k_t,U^{bk}_t,U^a_{t+1},\xi^k) + J^k_{t+1}(F_{t}(X^k_t,U^{bk}_t,U^a_{t+1},\xi^k),\theta^k_{t+1})\},\\
      X^k_{t+1} &= F_{t}(X^k_t,U^{bk}_t,U^{ak}_{t+1},\xi^k).
    \end{align*}
  \item
    {\bf Backward pass:} Let $J_T^{k+1}:=J$ and, for $t=T-1,\ldots,0$,
    \begin{align*}
      \tilde J_t^{k+1}(X_t^k,\theta_t^k)
      &:= \inf_{U^b_t\in\reals^{M^b}} E\left[\inf_{U^a_{t+1}\in\reals^{M^a}}\{L_t(X^k_t,U^b_t,U^a_{t+1}) + J^{k+1}_{t+1}(F_{t}(X^k_t,U^b_t,U^a_{t+1}))\}\Bigg{|}\theta^k_t\right],
      \\  
      V^{k+1}_t&\in\partial\tilde J^{k+1}_t(X_t^k,\theta_t^k)
    \end{align*}
    and let
    \[
      J_t^{k+1}(X_t,\theta_t^k) := \max\{J_t^k(X_t,\theta_t^k),\tilde J_t^{k+1}(X_t^k,\theta_t^k) + V_t^{k+1}\cdot(X_t-X^k_t)\}\quad\forall x_t\in\reals^N.
    \]
    Set $k:=k+1$ and go to 1.
  \end{enumerate}
}

Again, in the hazard-decision format, the second stage optimizations in the above two-stage problems can be done scenariowise. Moreover, in that case, the gradient in the backward pass is given by averaging the gradients of the infimum values with respect to $U^a_{t+1}$. Note also that in the decision-hazard format where the ``wait-and-see'' variables $U^a_t$ are absent, the optimization problems in the above control format of the SDDP algorithm become static problems of minimizing the expectations.

\bibliographystyle{plain}
\bibliography{sp}

\begin{thebibliography}{10}

\bibitem{bis73b}
J.-M. Bismut.
\newblock Conjugate convex functions in optimal stochastic control.
\newblock {\em J. Math. Anal. Appl.}, 44:384--404, 1973.

\bibitem{bis73}
J.-M. Bismut.
\newblock Int\'egrales convexes et probabilit\'es.
\newblock {\em J. Math. Anal. Appl.}, 42:639--673, 1973.

\bibitem{ccdmt}
P~Carpentier, J.-P. Chancelier, M.~De~Lara, T.~Martin, and T.~Rigaut.
\newblock Time block decomposition of multistage stochastic optimization
  problems.
\newblock {\em J. Convex Anal.}, 30(2):627--658, 2023.

\bibitem{dow20}
O.~Dowson.
\newblock The policy graph decomposition of multistage stochastic programming
  problems.
\newblock {\em Networks}, 76(1):3--23, 2020.

\bibitem{evs76}
I.~V. Evstigneev.
\newblock Measurable selection and dynamic programming.
\newblock {\em Math. Oper. Res.}, 1(3):267--272, 1976.

\bibitem{fl23}
M.~Forcier and V.~Lecl\`ere.
\newblock Trajectory following dynamic programming algorithms without finite
  support assumptions.
\newblock {\em J. Convex Anal.}, 30(3):951--999, 2023.

\bibitem{im96}
G.~Infanger and D.~P. Morton.
\newblock Cut sharing for multistage stochastic linear programs with interstage
  dependency.
\newblock {\em Math. Programming}, 75(2):241--256, 1996.

\bibitem{ls19}
N.~L{\"o}hndorf and A.~Shapiro.
\newblock Modeling time-dependent randomness in stochastic dual dynamic
  programming.
\newblock {\em European Journal of Operational Research}, 273(2):650--661,
  2019.

\bibitem{nes18}
Yu. Nesterov.
\newblock {\em Lectures on convex optimization}, volume 137 of {\em Springer
  Optimization and Its Applications}.
\newblock Springer, Cham, second edition, 2018.

\bibitem{pdccjr}
C.~Martinez Parra, M.~De~Lara, J.-P. Chancelier, P.~Carpentier, J.-M. Janin,
  and M.~Ruiz.
\newblock A two-timescale decision-hazard-decision formulation for storage
  usage values calculation.
\newblock {\em arXiv preprint arXiv:2408.17113}, 2024.

\bibitem{pp24}
T.~Pennanen and A.-P. Perkki{\"o}.
\newblock {\em Convex Stochastic Optimization}.
\newblock Springer Cham, 2024.

\bibitem{pp91}
M.~V.~F. Pereira and L.~M. V.~G. Pinto.
\newblock Multi-stage stochastic optimization applied to energy planning.
\newblock {\em Mathematical Programming}, 52(1):359--375, May 1991.

\bibitem{pm12}
A.B. Philpott and V.L. {de Matos}.
\newblock Dynamic sampling algorithms for multi-stage stochastic programs with
  risk aversion.
\newblock {\em European Journal of Operational Research}, 218(2):470--483,
  2012.

\bibitem{roc68}
R.~T. Rockafellar.
\newblock Integrals which are convex functionals.
\newblock {\em Pacific J. Math.}, 24:525--539, 1968.

\bibitem{roc70b}
R.~T. Rockafellar.
\newblock Conjugate convex functions in optimal control and the calculus of
  variations.
\newblock {\em J. Math. Anal. Appl.}, 32:174--222, 1970.

\bibitem{roc70a}
R.~T. Rockafellar.
\newblock {\em Convex analysis}.
\newblock Princeton Mathematical Series, No. 28. Princeton University Press,
  Princeton, N.J., 1970.

\bibitem{rw76}
R.~T. Rockafellar and R.~J.-B. Wets.
\newblock Nonanticipativity and {$L^1$}-martingales in stochastic optimization
  problems.
\newblock {\em Math. Programming Stud.}, (6):170--187, 1976.
\newblock Stochastic systems: modeling, identification and optimization, II
  (Proc. Sympos., Univ Kentucky, Lexington, Ky., 1975).

\bibitem{rw83}
R.~T. Rockafellar and R.~J.-B. Wets.
\newblock Deterministic and stochastic optimization problems of {B}olza type in
  discrete time.
\newblock {\em Stochastics}, 10(3-4):273--312, 1983.

\bibitem{rw98}
R.~T. Rockafellar and R.~J.-B. Wets.
\newblock {\em Variational analysis}, volume 317 of {\em Grundlehren der
  Mathematischen Wissenschaften [Fundamental Principles of Mathematical
  Sciences]}.
\newblock Springer-Verlag, Berlin, 1998.

\bibitem{sdr21}
A.~Shapiro, D.~Dentcheva, and A.~Ruszczy\'nski.
\newblock {\em Lectures on stochastic programming---modeling and theory},
  volume~28 of {\em MOS-SIAM Series on Optimization}.
\newblock Society for Industrial and Applied Mathematics (SIAM), Philadelphia,
  PA; Mathematical Optimization Society, Philadelphia, PA, third edition,
  [2021] \copyright 2021.

\bibitem{svlv}
A.~Street, D.~Valladão, A.~Lawson, and A.~Velloso.
\newblock Assessing the cost of the hazard-decision simplification in
  multistage stochastic hydrothermal scheduling.
\newblock {\em Applied Energy}, 280:115939, 2020.

\end{thebibliography}

\end{document}